\documentclass[a4paper,10pt]{amsart}

\usepackage{mathrsfs}
\usepackage[all]{xy}
\usepackage{calrsfs}
\usepackage{times}
\usepackage[scaled=0.92]{helvet}
\usepackage{wrapfig}

\usepackage{amsfonts}
\usepackage{amssymb}
\usepackage[usenames]{color}
\usepackage{todonotes}
\usepackage[usenames]{color}
\newtheorem{theorem}{Theorem}   

\newtheorem{lemma}[theorem]{Lemma}
\newtheorem{proposition}[theorem]{Proposition}

\theoremstyle{definition}
\newtheorem{definition}[theorem]{Definition}
\newtheorem{remark}[theorem]{Remark}

\newtheorem{Que}[theorem]{Question}

\definecolor{MyDarkBlue}{rgb}{0,0.08,0.60}

\newcommand{\A}{{ \rm Aut }}

\newcommand{\Z}{{\mathbb{Z}}}

\renewcommand{\mod}{{\;\rm mod}}

\newcommand{\Q}{{\mathbb{Q}}}
\newcommand{\C}{{\mathbb{C}}}

\usepackage{color, colortbl}
\definecolor{LightCyan}{rgb}{0.88,1,1}
\definecolor{Gray}{gray}{0.9}

\usepackage{changes}
\definechangesauthor[color=orange,name={Aristides Kontogeorgis}]{AK}

\title
[Automorphisms of the Heisenberg curve]
{
The group of automorphisms of the Heisenberg curve
}

\date{\today}

\author{Jannis A. Antoniadis }
\address{Department of Mathematcis, University of Crete\\
Herakleion Crete Greece 71409}
\email{antoniad@uoc.gr}

\author{Aristides Kontogeorgis}
\address{Department of Mathematics, University of Athens\\
Panepistimioupolis, 15784 Athens, Greece
}
\email{kontogar@math.uoa.gr}

\begin{document}
\bibliographystyle{amsplain}

\begin{abstract}
The Heisenberg curve is defined to be the curve corresponding to an extension of the projective line by the Heisenberg group modulo $n$, ramified above three points. This curve is related to the Fermat curve and its group of automorphisms is studied. 
Also we give an explicit equation for the curve $C_3$.
\end{abstract} 

\thanks{{\bf keywords:} Automorphisms, Curves, Differentials, Numerical
Semigroups.
 }

\maketitle 
\section{Introduction}
Probably the most famous curve in number theory is the Fermat curve given by affine equation 
\[
 F_n: x^n+y^n=1.
\]
This curve can be seen as ramified Galois cover of the projective line with Galois  
group $\mathbb{Z}/n\mathbb{Z} \times \mathbb{Z}/n \mathbb{Z}$ with action given by 
$\sigma_{a,b}:(x,y) \mapsto (\zeta^a x, \zeta^b y)$ where
$(a,b) \in \mathbb{Z}/n\mathbb{Z} \times \mathbb{Z}/n \mathbb{Z}$.
The   ramified  cover
 \[\pi:F_n\rightarrow \mathbb{P}^1\]
 has three ramified points 
and the cover 
\[
F_n^0:=F_n- \pi^{-1}(\{ 0,1,\infty \}) \stackrel{\pi}\longrightarrow 
\left(
\mathbb{P}^1 - \{0,1,\infty\}
\right),
\]
is a Galois topological cover. We can see the hyperbolic space $\mathbb{H}$ as the universal 
covering space of $\mathbb{P}^1 - \{0,1,\infty\}$.
The Galois group of the above cover is isomorphic to the free group $F_2$  in two generators,
and a suitable realization of this group in our setting is the group 
$
\Delta
$ 
which is the subgroup of $\mathrm{SL}(2,\Z) \subseteq \mathrm{PSL}(2,\mathbb{R})$  generated by the elements
$a=\begin{pmatrix} 1 & 2 \\ 0 & 1 \end{pmatrix}$, $b=\begin{pmatrix} 1 & 0 \\ 2 & 1\end{pmatrix}
$ and $\pi(\mathbb{P}^1 - \{0,1,\infty\},x_0)\cong \Delta$.
Related to the group $\Delta$ is the modular group
\[
\Gamma(2)=\left\{ \gamma=
\begin{pmatrix}
 a & b \\
 c & d
\end{pmatrix} \in \mathrm{SL}(2,\mathbb{Z}): \gamma \equiv 1_2 \mod 2
\right\},
\]
which is isomorphic to $\{\pm I\}\Delta$
 while 
$\Gamma(2) \backslash \mathbb{H} \cong 
\left(\mathbb{P}^1 - \{0,1,\infty\}\right)$.
The groups $\Delta$ and $ \Gamma(2)$ act in exactly the same way on the hyperbolic plane
$\mathbb{H}$. 

\begin{remark}
Covers of the projective line minus three points are very important 
in number theory because of the Belyi theorem \cite{Belyi1},\cite{Belyi2} that asserts that all algebraic curves 
defined over $\bar{\mathbb{Q}}$ fall into this category. It seems that the idea of studying algebraic curves as ``modular curves'' goes back to S. Lang and to D. Rohrlich \cite{MR2626317}.
\end{remark}
For every finitely generated group  $G$ generated by two elements there is a homomorphism 
\[
 \Gamma(2) \rightarrow G.
\]
Notice that $\Gamma(2)^{\mathrm{ab}} \cong \mathbb{Z} \times \mathbb{Z}$ so 
using the projection 
\[
 \psi:\Delta \rightarrow \Delta^{\mathrm{ab}} \rightarrow \mathbb{Z}/n\mathbb{Z} \times \mathbb{Z}/n \mathbb{Z}
\]
we can write the open Fermat curve $F^0_n$ as the quotient 
\[
 F^0_n= \ker \psi \backslash \mathbb{H}.
\]
In the theory of modular curves \cite{milneMF}, the hyperbolic space $\mathbb{H}$ is extented to $\bar{\mathbb{H}}=\mathbb{H} \cup \mathbb{P}^1(\Q)$, where $\mathbb{P}^{1}(\Q)$ is the set of cusps so that subgroups $\Gamma$ of  $\mathrm{SL}(2,\Z)$ give rise to compact quotients. The orbits of $\mathbb{P}^1(\Q)$
under the action of $\Gamma$ are the cusps of the curve $\Gamma\backslash \mathbb{H}$. In this setting the cusps of the Fermat curve $F_n$ are the  points $F_n-F_n^0$.

Aim of this article is to initialize the study of the curve $C_n$, which is  defined in the following way:
Consider the Heisenberg group modulo $n$:
\[
 H_n=\left\{ 
 \begin{pmatrix}
  1 & x & z \\
  0 & 1 & y \\
  0 & 0 & 1
 \end{pmatrix}:
 x,y,z \in \mathbb{Z}/n\mathbb{Z}
 \right\}.
\]
It is a finite group generated by the elements
\begin{equation} \label{a-b}
 a_H=\begin{pmatrix} 1 & 1 & 0 \\
 0 & 1 & 0 \\
 0 & 0 & 1 
 \end{pmatrix}
 \mbox{ and }
 b_H=\begin{pmatrix} 1 & 0 & 0 \\
 0 & 1 & 1 \\
 0 & 0 & 1 
 \end{pmatrix}.
\end{equation}
Notice that 
\begin{equation} \label{commutator}
 [a_H,b_H]=\begin{pmatrix} 1 & 0 & 1 \\
 0 & 1 & 0 \\
 0 & 0 & 1 
 \end{pmatrix}
\end{equation}
and 
\[
 H_n^{\mathrm{ab}} \cong \mathbb{Z}/n\mathbb{Z} \times \mathbb{Z}/n \mathbb{Z}.
\]
We have the short exact sequence
\begin{equation} \label{quo}
 1\rightarrow \mathbb{Z}/n\mathbb{Z} \cong Z_n \rightarrow H_n 
 \longrightarrow H_n^{\mathrm{ab}} \rightarrow 1,
\end{equation}
where 
\[
 Z_n:=\left\{
  \begin{pmatrix} 1 & 0 & z \\
 0 & 1 & 0 \\
 0 & 0 & 1 
 \end{pmatrix}
 :z \in \mathbb{Z}/n\mathbb{Z}
 \right\}.
\]
Observe that there is an epimorphism: $\phi:\Gamma(2) \rightarrow H_n$ sending each 
generator of $\Gamma(2)$ to the elements $a_H,b_H\in H_n$. 
This way an open curve $C_n^0$ is defined as $\ker \phi \backslash \mathbb{H}$ that 
can be compactified to a compact Riemann surface $C_n$. 

We have the following diagram:
\begin{equation}
\label{open-uniformization}
 \xymatrix{
 \mathbb{H} \ar[d]^{\ker \phi} \ar@/_3pc/[ddd]_{\Gamma(2)}  \ar@/^3pc/[dd]^{\ker \psi}\\
 C_n \ar[d]^{Z_n} \ar@/_1pc/[dd]_{H_n} \\
 F_n \ar[d]^{\mathbb{Z}/n\mathbb{Z}\times \mathbb{Z}/n\mathbb{Z}}  \\
 \mathbb{P}^1
 }
\end{equation}

\begin{definition}
 Let $X$ be a curve that comes as a compactification by adding some cusps  of the open curve
 $\Gamma \backslash \mathbb{H}$ where $\Gamma$ is a subgroup of finite index of $\mathrm{SL}(2,\Z)$. 
 The group of {\em modular automorphisms} is the group
 \[
  \A^\mathrm{m} (X) = N_{\mathrm{SL}(2,\mathbb{R})}(\Gamma)/\Gamma,
 \]
 where $\mathrm{SL}(2,\mathbb{R})$ is the group of automorphisms of $\mathbb{H}$ and 
 $N_{\mathrm{SL}(2,\mathbb{R})}(\Gamma)$ is the normalizer of $\Gamma$ in $\mathrm{SL}(2,\mathbb{R})$. 
\end{definition}

\begin{remark}
An automorphism $\sigma$  of a complete curve $X$ which comes out from an open curve $\Gamma \backslash \mathbb{H}$ by adding the set of cusps $\mathbb{P}^1(\mathbb{Q})$  is modular if and only if $\sigma $ sends cusps
 to cusps and non-cusps to non-cusps.
\end{remark}

\begin{remark}
Deciding if there are extra non-modular automorphism is a difficult classical question
for the case of modular and Shimura curves, see 
 \cite{Baker2003-new},
\cite{Kenku1988-xb}, \cite{Kamienny2003-tg}, \cite{Harrison2011-ke}, \cite{AKRo1}, \cite{Molina2013-ni}, \cite{KoYang} for some related results.  
\end{remark}

\noindent
{\bf Acknowledgement: }
The authors would like to thank Professor Dinakar Ramakrishnan
for proposing the study of Heisenberg curves to them.

\section{The triangle group approach}
\label{sec:triangle}

A Fuchsian group $\Gamma$ is a finitely generated discrete subgroup of $\mathrm{PSL}(2,\mathbb{R})$. It is known that  a Fuchsian group has a set of $2g$ hyperbolic generators 
$\{a_1,b_1,\ldots,a_g,b_g\}$, a set of elliptic generators $x_1,\ldots,x_r$ and parabolic generators $p_1,\ldots,p_s$ and some hyperbolic boundary elements 
$h_1,\ldots,h_t$, see \cite{Singerman1972-vp}. The relations are given by 
\[
x_1^{m_1}=x_2^{m_2}=\cdots=x_r^{m_r}=1
\]
\[
\prod_{i=1}^{g} [a_i,b_i] \prod_{j=1}^r x_j \prod_{k=1}^s p_k \prod_{i=1}^t h_t=1.\]
The signature of $\Gamma$ is 
\[
(g; m_1,\ldots,m_r;s;t),
\]
where $m_1,\ldots,m_r$ are natural numbers   $\geq 2$ and are called the periods of $\Gamma$. 

{\bf Triangle groups:} A triangle group $\Delta(\ell,m,n)$ is a group with signature $[0;\ell,m,n]$.
We also thing parabolic elements as elliptic elements of infinite period and in this point of view , the group $\Gamma(2)$ can also be considered as the triangle group $\Gamma(\infty,\infty,\infty)$. 


The Fermat curve can be uniformized in terms of triangle groups. This is a quite different uniformization than the uniformization given 
in (\ref{open-uniformization}). 
Namely we have the following diagramm of curves and groups
\begin{equation} \label{triangle}
\xymatrix{
  \mathbb{H} \ar[d] \ar@/_6.5pc/[ddd]_{\Delta(n,n,n)}  \ar@/^3pc/[dd]^{\Delta(n,n,n)'}\\
 C_n 
 \ar[d]^{Z_n} \ar@/_4pc/[dd]_{H_n} 
 \\
 F_n\cong  \Delta(n,n,n)' \backslash \mathbb{H} \ar[d]^{\mathbb{Z}/n\mathbb{Z}\times \mathbb{Z}/n\mathbb{Z}}  \\
 \mathbb{P}^1\cong  \Delta(n,n,n) \backslash \mathbb{H}
}
\end{equation}
 $\Delta(n,n,n)'$ is the commutator of the triangle group $\Delta(n,n,n)$.
For $n>3$
 it is known that $\Delta(n,n,n)'$ is the universal covering group of the Fermat curve  see \cite{MR1483112}.

We will show in lemma \ref{11} that if $(n,2)=1$,  then $C_n\rightarrow F_n$ is unramified. In this case, if $D(n)$ denotes the universal covering group of the Heisenberg curve then $D(n)$ is a normal subgroup of $\Delta(n,n,n)' 
$ and
$\Delta(n,n,n)'/D(n)\cong \Z/n\Z$.

The presentation of a Riemann surface in terms of a co-compact triangle group has several advantages. Concerning automorphism groups, the advantage is that if $\Pi$ is the funtamental group of the curve, which is a normal torsion free subgroup of the triangle group $\Delta(a,b,c)$, then the group $N_{\mathrm{PSL}(2,\mathbb{R})} (\Pi)/\Pi$ is the whole automorphism group not only the group of modular automorphisms, see \cite{Kallel2005-ru}.  The computation of the automorphism group is then simplified, as we will see for the case of Heisenberg and Fermat curves, since
 $N_{\mathrm{PSL}(2,\mathbb{R})}(\Pi)$ is known to  be also a triangle curve which contains $\Delta(\ell,m,n)$, 
and these groups are fully classified \cite{Greenberg1963-tf}, \cite{Singerman1972-vp}, \cite[table 2]{Kallel2005-ru}. This approach has the disadvantage that does not provide explicitly the automorphisms acting on the curve. 

\begin{remark}
\label{remarkFermatAut}
The automorphism group of the Fermat curve can be computed by using the classification of the triangle groups which contain $\Delta(n,n,n)$ as a normal subgroup \cite{Greenberg1963-tf}, \cite{Singerman1972-vp}, \cite[table 2]{Kallel2005-ru}. Indeed, the only such group is $\Delta(2,3,2n)$ and this computation provides and alternative method for proving:
\[\mathrm{Aut}(F_n)=\Delta(2,3,2n)/\Delta(n,n,n)= \left(\mathbb{Z}/n \mathbb{Z} \times \mathbb{Z}/n \mathbb{Z} \right) \rtimes  S_3.\]
Notice that the triangle curve $\Delta(2,3,2n)$ is compatible with the ramification diagram given in figure \ref{FigFermat}.
\end{remark}
\begin{remark}
In addition to the above proof, 
the authors are aware of the following different methods for computing  the automorphisms group of the Fermat curve: there are proofs using the Riemann-Hurwitz formula \cite{Tze:95}, \cite{Leopoldt:96} and  proofs using the embedding of the Fermat curve in $\mathbb{P}^2$ and projective duality \cite{Shioda:87}, \cite{Kontogeorgis2002-to}. 
\end{remark}


\section{Restriction and Lifting of automorphisms in covers}
In this section we will study the following:

\begin{Que}
Assume that $X\rightarrow Y$ is a Galois cover of curves. How are the groups $\A(X)$ and $\A(Y)$ related?
When can an automorphism in $\A(Y)$ lift to an automorphism of $\A(X)$?
\end{Que}

Assume that 
$X^0=\Gamma_X \backslash \mathbb{H},Y^0=\Gamma_Y \backslash \mathbb{H}$ are either open curves corresponding to certain subgroups
 $\Gamma_X,\Gamma_Y$ of $\mathrm{SL}(2,\Z)$ or complete curves uniformized by cocompact groups $\Gamma_X,\Gamma_Y$. In the first case the complete curves $X,Y$ are obtained by adding the cusps and in the second case $X^0=X$ and $Y^0=Y$.

\begin{proposition}
\label{rest-ext}
Let $\{1\} < \Gamma_X \lhd \Gamma_Y$  and consider the sequence of Galois 
covers
\[
\mathbb{H} \rightarrow \Gamma_X \backslash \mathbb{H} \rightarrow \Gamma_Y \backslash \mathbb{H} 
\] 
and $G=\mathrm{Gal}(X/Y)$.
If $\Gamma_X, \Gamma_Y$ are both normal subgroups of $\mathrm{SL}(2,\Z)$ then  a modular automorphism $\sigma$ of $Y$ lifts to $|G|$  automorphisms of $X$, if and 
only if $\sigma \Gamma_X \sigma^{-1} \subset \Gamma_X$. A modular automorphism $\tau$ of $ X$ restricts to a modular automorphism of $Y$ if and only if $\tau G \tau^{-1} \subset G$, i.e. if and only if $\tau^{-1} \Gamma_Y \tau \subset \Gamma_Y$.  

Similarly if $\Gamma_X,\Gamma_Y$ are cocompact subgroups of $\mathrm{SL}(2,\mathbb{R})$ uniformizing the compact curves $X,Y$,  then an  automorphism $\sigma$ of $Y$ lifts to $G$  automorphisms of $X$, if and 
only if $\sigma \Gamma_X \sigma^{-1} \subset \Gamma_X$. An automorphism $\tau$ of $ X$ restricts to a modular automorphism of $Y$ if and only if $\tau G \tau^{-1} \subset G$, i.e. if and only if $\tau^{-1} \Gamma_Y \tau \subset \Gamma_Y$.  
\end{proposition}
\begin{proof}
We will prove the first case, where $\Gamma_X,\Gamma_Y$ are subgroups of $\mathrm{SL}(2,\Z)$ uniformizing the open curves $X^0,Y^0$. The second case has a  similar proof.

By definition, a modular automorphism of $Y$, is represented by element in the normalizer $N_{\mathrm{SL}(2,\mathbb{R})}\Gamma_Y$. This element has to normalize $\Gamma_X$ as well in order to extend to an automorphism of $\Gamma_X$. 

\noindent
\begin{minipage}{0.75\textwidth}
For restricting an automorphism from $X$ to $Y$. 
Set 
$G=\mathrm{Gal}(X/Y)=\Gamma_Y/\Gamma_X$ and let $N$ be the subgroup of $\mathrm{Aut}^0(X)$ which restrict to automorphisms of $Y$. We have the tower of fields shown on the right. The group $N$ has to normalize
 $G$ so that $N/G$ is group acting on $Y$. Since $G$
is a subgroup of the modular automorphism group, there is a conjugation action of every automorphism of $X$ on $G$. In particular a modular automorphism $\tau\cdot\Gamma_X$ of $X$ acts by conjugation on an element  $ \gamma \cdot \Gamma_X \in G$:
\end{minipage}
\begin{minipage}{0.25\textwidth}
\[
\xymatrix{
	X \ar[d]_G \ar@/^2pc/[dd]^N
	\\
	Y \ar[d]_{N/G} 
	\\
	N \backslash X
}
\]
\end{minipage}

\[
(\tau\cdot\Gamma_X)( \gamma\cdot\Gamma_X) (\tau\cdot\Gamma_X)^{-1}=
\tau \gamma \tau^{-1} \cdot \Gamma_X,
\]
and the later element is in $G$ if an only if $\tau \gamma \tau^{-1} \in \Gamma_Y$. 
\end{proof}

\section{The Fermat curve}

In this section we will collect some known results about the Fermat curve and 
its automorphism group. 

We will use the coorespondence of functions fields of one variable to curves and of points to places, see \cite{Stichtenothv2009}. In particular we will use that coverings of curves correspond to algebraic extensions of their function fields. 

\begin{lemma}
For $n\geq 4$
the automorphism group of the Fermat curve is isomorphic to 
 the semidirect product $(\Z/n\Z \times \Z/n\Z) \rtimes S_3$. 
  In the cover $F_n \rightarrow F_n/\A(F_n)\cong \mathbb{P}^1$,
  three points are ramified with ramification indices $2n,3,3$ respectively.  
\end{lemma}
\begin{proof}
For the  automorphism group of the Fermat curve see \cite{Tze:95},\cite{Leopoldt:96}. In characteristic zero the automorphism group can be also studied as in remark \ref{remarkFermatAut}. 

 The curves $F_2$ and $F_3$ are rational and elliptic respectively and so they have infinite 
 automorphism group. The Fermat curve can be seen as a Kummer cover, i.e. the 
 function field $\mathbb{C}(F_n)=\mathbb{C}(x)[\sqrt[n]{x^n-1}]$ is a Kummer extension of
 the rational 
 function field $\mathbb{C}(x)$ and the ramification places in this Kummer extension 
 correspond to the irreducible polynomials $x-\zeta^i$, $i=0,\ldots,n-1$, where 
 $\zeta$ is a primitive $n$-th root of unity. We have the following picture of 
 function fields and ramification of places.
 
\begin{figure}[h] 
 {\small
 \[
  \xymatrix@=0.8pc{
  \mathbb{C}(F_n) \ar@{-}[d]_{\Z/n\Z} 
  & P_0, \ar@{-}[d]_{n} 
   & P_{n-1} \ar@{-}[d]_{n} & Q_0,\ldots,Q_{n-1} \ar@{-}[d] & Q_0',\ldots,Q_{n-1}' \ar@{-}[d]\\
  \mathbb{C}(x) \ar@{-}[d]_{\Z/n\Z}   
  & P_{(x=1)} \ar@{-}[rd]    & P_{(x=\zeta^{n-1})} \ar@{-}[d] & Q_0 \ar@{-}[d]_{n}& 
  Q_\infty \ar@{-}[d]_{n}\\
  \mathbb{C}(x^n) \ar@{-}[d]_{S_3} & &  P_{(x^n=1)} \ar@{-}[d]^2     & P_{(x^n=0)} 
  \ar@{-}[ld]_2
   & P_\infty \ar@{-}[lld]^2  &P_1,P_2, P_3 \ar@{-}[d]_{2} & P_1',P_2'\ar@{-}[d]_{3} \\
  \mathbb{C}(t)   & &  p  & & &  p_1 & p_2&
  }
 \]
}
\caption{Ramification diagram for the Fermat Curve}\label{FigFermat}
\end{figure}
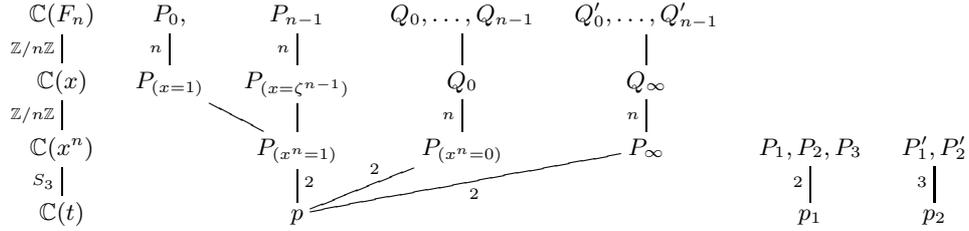

\end{proof}

Let us now consider the cover $F_n \rightarrow F_n^{\Z/n\Z \times \Z/n\Z} \cong \mathbb{P}^1$.
Consider the function field $\C(x^n)=\C(F_n)^{\Z/n\Z \times \Z/n\Z}$.
In the extension $\C(F_n)/\C(x^n)$
three places  are ramified $P_{(x^n=0)},P_{(x^n=1)},P_\infty$. We will describe now the places  of 
$\C(F_n)$ which restrict to the three points  ramified above. For this we need the projective form of the Fermat 
curve given by
\[
 X^n+Y^n=Z^n.
\]
The $n$ places $P_0,\ldots,P_{n-1}$ which  restrict to $P_{(x^n=1)}$ correspond to points with   projective coordinates 
$P_k=(\zeta^k:0:1)$ for $k=0,\ldots, n-1$. 
We have also the $n$ places $Q_0,\ldots,Q_{n-1}$
which  restrict to $P_{x^n=0}$ which correspond to points with  projective coordinates  
$Q_k=(0:\zeta^k:1)$ and the $n$ places  $Q'_0,\ldots,Q'_{n-1}$ that restrict to $P_\infty$ and correspond to points with  projective coordinates
$Q'_k=(\epsilon \zeta^k:1:0)$, $k=0,\ldots,n-1$, $\epsilon^2=\zeta$. An element  $\sigma_{a,b} \in \Z/n\Z \times \Z/n\Z$
is acting on coordinates $(X:Y:Z)$ by the following rule:
\[
 \sigma_{a,b}:(X,Y,Z) \mapsto (\zeta^a X,\zeta^b Y ,Z).
\]
We will  compute the stabilizers of the places $P_0,\ldots,P_{n-1}$, $Q_0,\ldots,Q_{n-1}$, $Q_0',\ldots, Q_{n-1}'$ ramified in 
$\C(F_n)/\C(x^n)$:
\begin{lemma} \label{stab}
 The points $(\zeta^k:0:1)$ for $k=0,\ldots,n-1$   are fixed by the cyclic group of order $n$
 generated by $\sigma_{0,1}$. The points $(0:\zeta^k:1)$, $k=0,\ldots, n-1$ are fixed by 
 the cyclic group of order $n$ generated by $\sigma_{1,0}$. Finaly the points 
 $(\zeta^k:1:0)$ for $k=0,\ldots,n-1$ are fixed by the cyclic group of order 
 $n$ generated by $\sigma_{1,1}$. 
\end{lemma}
\begin{proof}
 By computation. 
\end{proof}

\section{Automorphism Group of the Heisenberg curve}

\begin{lemma} \label{11}
If $(n, 2)=1$, 
then 
 the cover $C_n  \rightarrow F_n$ is unramified. 
\end{lemma}
\begin{proof}
Notice first that
\begin{equation} \label{power-comp}
 \begin{pmatrix}
  1  & x & z \\
  0  & 1 & y \\
  0  & 0 & 1
 \end{pmatrix}^\nu=
 \begin{pmatrix}
  1  & \nu x & \nu z + \frac{\nu(\nu-1)}{2}xy\\
  0  & 1 & \nu y \\
  0  & 0 & 1
 \end{pmatrix}.
\end{equation}
 If $(n,2)=1$ then by eq. (\ref{power-comp})  
every element in $H_n$ has order at most $n$. 
Notice that  the  top right corner element is 
$\nu z+\frac{\nu(\nu-1)}{2}xy$ and $2 \mid (n-1)$.


Also  the only points that can 
ramify in $C_n \rightarrow F_n$ are the points of the Fermat curve, which  lie 
above $\{0,1,\infty\}$ since outside this set the cover is unramified. 
But if such a point $P$ of the Heisenberg curve was ramified in $C_n \rightarrow F_n$, then its stabilizer $H_n(P)$ should be a cyclic group of 
order greater that $n$ and no such group exist. 
%
%
%
\end{proof}

Assume now that $(n,2)=2$. Consider the map $\pi: H_n \rightarrow H_n^{\mathrm{ab}}\cong \mathbb{Z}/n\mathbb{Z} \times \mathbb{Z}/n\mathbb{Z}$ defined 
in equation (\ref{quo}). Let $\sigma_{a,b} \in H_n^{\mathrm{ab}}$ be the automorphism corresponding to the pair $(a,b)\in \mathbb{Z}/n\mathbb{Z} \times \mathbb{Z}/n\mathbb{Z}$ with $a,b \in \mathbb{Z}/n\mathbb{Z}$.

Select an element $\alpha_{a,b}$ such that $\pi(\alpha)=\sigma_{a,b}$. Such an element has the following matrix form for some $z\in \Z/n\Z$:
\[
 \alpha_{a,b}=
 \begin{pmatrix}
  1 & a & z \\
  0 & 1 & b \\
  0 & 0 & 1
 \end{pmatrix}.
\]
By eq. (\ref{power-comp}) if $ab=0$ and $a=1$ or $b=1$ then the order $\mathrm{ord}(\alpha_{a,b})$ is at most  $n$. 
So  $\mathrm{ord}(\alpha_{1,0})=n$ and $\mathrm{ord}(\alpha_{0,1})=n$. 
Using again (\ref{power-comp}) we see that
  $\mathrm{ord}(\alpha_{1,1})=2n$.

This fact, combined to the computation of the stabilizers given in lemma \ref{stab} 
gives the following:

\begin{lemma} \label{ram-al}
If $(n,2)=2$ in   the cover $C_n \rightarrow F_n$ only the points 
$(\zeta^k:1:0)$ above $P_\infty$ can ramify with ramification index at 
most $2$.
\end{lemma}
In order to understand the even $n$ case we will treat first the  $n=2$ case.
\begin{lemma}
 The Heisenberg curve $C_2$ is rational and in $C_2 \rightarrow F_2$ only two 
 points of $F_2$ are branched in the cover $C_2\rightarrow F_2$, namely $(1:1:0)$ and $(-1:1:0)$. 
\end{lemma}
\begin{proof}
The case $n=2$ is special since the Fermat curve $F_2$ is rational.
In this case the Heisenberg  group $H_2$  has order $8$ and 
is isomorphic to the dihedral group $D_4$ 
generated by the elements
\[
a_1=
\begin{pmatrix}
 1 & 1 & 1 \\
 0 & 1 & 1 \\
 0 & 0 & 1
\end{pmatrix}  \mod 2
\mbox{ and }
a_2=
\begin{pmatrix}
 1 & 1 & 0 \\
 0 & 1 & 0 \\
 0 & 0 & 1
\end{pmatrix} \mod 2,
\]
where the order of $a_1$ is $4$ and the order of $a_2$ is $2$.
From the classification of finite subgroups of the projective line 
\cite{Val-Mad:80},
we have the following: 
The curve $F_2$ has  
the group $D_2=\Z/2\Z \times \Z/2\Z$ inside its automorphism group, while in the cover  $F_2 \rightarrow F_2/{D_2}$
three points $\{0,1,\infty\}$ are ramified with ramification indices $2$.
At least one point should ramify in the cover $C_2 \rightarrow F_2$ since 
$F_2$ is simply connected, this should be a cusp and the only cusps that are permitted 
by lemma \ref{ram-al} are $(1:1:0)$, $(-1:1:0)$ and both should ramifify. 
The genus of $C_2$ is zero by Riemann-Hurwitz formula. The ramification indices in the intermediate extensions are shown in the following table:
{\small
 \[
  \xymatrix@=1pc{
  C_2 \ar[d]^{\Z/2\Z}  \ar@/_2pc/[dd]_{D_4}
  &  &  & Q_1  \ar@{-}^{2}[d] & Q_2 \ar@{-}^{2}[d] \\
  F_2 \ar[d]^{D_2}   
    & P_1,P_2  \ar@{-}[d]_{2}& 
  P_3,P_4 \ar@{-}[d]_{2} & P_5  \ar@{-}[d]_{2}    &  P_6 \ar@{-}[dl]_{2}\\
  D_2 \backslash F_2   &  1       & 0  & \infty   & 
  }
 \]
}
Notice that the ramification type of $D_4$ acting on the rational function field is $(2,2,4)$, which is in accordance to the classification in \cite{Val-Mad:80}. 
%
%
%
\end{proof}
\begin{lemma}
If $2\mid n$ then the 
 cusps of $C_n$ of the form $(\zeta^k:1:0)$ are ramified in the 
cover $C_n \rightarrow F_n$ with ramification index equal to two. 
\end{lemma}
\begin{proof}

Observe now that the elements of order $2n$
\[
\sigma_j:=\begin{pmatrix}
 1 & 1 & j \\
 0 & 1 & 1 \\
 0 & 0 & 1
\end{pmatrix}
\]
all restrict to $\sigma_{1,1} \in \mathrm{Aut}(F_n)$.
Consider the points $Q_1',\ldots,Q_n'$ of $F_n$ that are above $\infty$ and for a fixed point $Q_{\nu_0}'$ consider the set of elements $\bar{Q}_{\nu_0,j}$ $j=1,\ldots,t$ extending $Q'_{\nu_0}$ for $i=1,\ldots,n$.
Select now  a point $\bar{Q}_{\nu_0,\mu_0}$ among them. If $\sigma_0$ does not fix $\bar{Q}_{\nu_0,\mu_0}$ the it moves it to the point $\bar{Q}_{\nu_0,j}$. But then there is an element $\tau \in \mathrm{Gal}(C_n,F_n)=Z_n$ moving $\bar{Q}_{\nu_0,\mu_0}$. Therefore its fixed by a matrix $\sigma_j$ for an appropriate $j$. We compute 
\[
\sigma_j^{n}=
\begin{pmatrix}
 1 & 1 & j \\
 0 & 1 & 1 \\
 0 & 0 & 1
\end{pmatrix}^n
=
\begin{pmatrix}
 1 & 0 & -n/2 \\
 0 & 1 & 0 \\
 0 & 0 & 1
\end{pmatrix} \in Z_n.
\]
This means that in the cover $C_n \rightarrow (\Z/n\Z \times \Z/n\Z) \backslash F_n$ the point $\bar{Q}_{\nu_0,\mu_0}$ is ramified with ramification index $2n$. 



\end{proof}

\begin{lemma}
 The genus $g_{C_n}$ of the curve $C_n$ equals:
 \[
  g_{C_n}=\left\{
  \begin{array}{ll}
   \frac{n^2(n-3)}2+1 & \mbox{ if } (n,2)=1  \\
   \frac{n^2(n-3)}2 + \frac{n^2}{4}+1                  & \mbox{ if } 2 \mid n
  \end{array}
  \right.
 \]
\end{lemma}
\begin{proof}
 If $(n,2)=1$, then the cover $C_n \rightarrow F_n$ is unramified with Galois
 group $Z_n=\Z/n\Z$ and Riemann-Hurwitz formula implies that 
 \[
  2g_{C_n}-2=n(2g_{F_n}-2).
 \] 
 We know that $g_{F_n}=\frac{(n-1)(n-2)}{2}$ and this gives the result in this case. 
 
 If $2 \mid n$ then in the cover $C_n \rightarrow F_n$  $n$ cusps of the 
 Fermat curve are ramified with ramification index $2$. These cusps have 
 $\frac{ n^2}{2}$ points  in total above them, so Riemann-Hurwitz in this case gives
 \begin{equation}
 \label{g2Cn}
  2g_{C_n}-2=n (2g_{F_n}-2) +\frac{n^2}{2},
 \end{equation}
and the desired result follows.
\end{proof}

We have the following diagram:
\[
\xymatrix{
C_n \ar[d] &   P_{1,j},\ldots P_{n,j}  \ar@{-}[d] &   Q_{1,j},\ldots Q_{n,j}  \ar@{-}[d] &  Q'_{1,j},\ldots Q'_{n/e,j}  \ar@{-}^{e= 1 \mbox{ or } 2} [d]  \\
F_n \ar[d] &   P_j \ar@{-}[d] & Q_ j  \ar@{-}[d]  &  Q_j' \ar@{-}[d] \\
\mathbb{P}^1\cong \frac{F_n}{\mathbb{Z}/n\mathbb{Z} \times \mathbb{Z}/n\mathbb{Z}}  & 0 & 1 & \infty
}
\]
where $j=0,\ldots,n-1$ and $e$ is $1$ or $2$ according to the value of $n \mod 2$.

\section{Modular automorphisms}

\subsection{The Fermat curve}
 The open Fermat curve is the curve $F_n^0:=\ker \psi
 \backslash \mathbb{H}$, 
 where 
 \[
 \mathrm{ker}  \psi:=\langle a^n,b^n ,[a,b] \rangle\subset \langle a,b \rangle=\Gamma(2). 
 \]
 Notice that the above group differs from the universal covering group $\Delta(n,n,n)'$ given in section \ref{sec:triangle} which correspond to the closed Fermat curve.  

 Every automorphism of the Fermat curve is modular. 
 Indeed, the generators $\sigma_{1,0}$, $\sigma_{0,1}$ of the $\Z/N\Z \times \Z/N\Z$
 part of the automorphism group, are coming from the deck transformations 
 $a,b \in 
 \pi^1(\mathbb{P}^1 \backslash \{0,1,\infty\})
 =\Delta \subset \mathrm{SL}(2,\mathbb{Z}) \subset  \mathrm{SL}(2,\mathbb{R})$. 
 On the other hand it is known \cite[exer. 3 p. 32]{MR1453580} that 
   $S_3=\mathrm{SL}(2,\Z)/\Gamma(2)$ and the action is given by lifts of elements of $S_3$ to 
 $\mathrm{SL}(2,\Z)$. The fact that all automorphisms are indeed modular comes from the fact that $\mathrm{SL}(2,\Z)$ leaves both the set $\mathbb{H}$ and the cusps $\mathbb{P}^1$ invariant.

Let $\mathbf{F}_n$ denote the free group in $n$ generators. Notice that the  group $S_3$ acts by conjugation on $\Delta$ so it can be seen as 
 a subgroup of the group of outer automorphisms of $\Delta$. 
 It is known \cite[exam. 1 p. 117]{BirmanBraids}, \cite[th. 3.1.7 p. 125]{bogoGrp} that the  epimorphism 
 \[
\mathbf{F}_2 \rightarrow \mathbf{F}_2^{\mathrm{ab}}\cong \Z \times \Z
 \]
 induces an isomorphism of 
 \[
  \mathrm{Out}(\mathbf{F}_2)/
  \mathrm{In}(\mathbf{F}_2) =\A (\mathbf{F}_2)  \rightarrow \mathrm{GL}(2,\Z).
 \]
The group $S_3$ is generated by the 
folowing automorphisms of the free group $\mathbf{F}_2=\langle a,b\rangle$:
\begin{equation}
\label{i-involution}
i_1: a\leftrightarrow b \qquad i_2: 
\begin{array}{l}
a \mapsto b^{-1} a^{-1}\\
b \mapsto b
\end{array}
\end{equation}
The above generators $i_1,i_2$ of $S_3$ keep the group $\mathrm{ker}\psi$ invariant.

On the other hand, an arbitrary element 
of $S_3$ reduces to an action on the Fermat curve $F_n$ by permutation of the variables $X,Y,Z$ in the projective model of the 
curve.    
The automorphism intechanging $X,Y$ coressponds to the involution interchanging $a,b$.
Let us now consider the automorphism $\bar{\tau}$ interchanging $X,Z$ in the projective model the 
Fermat curve.

Denote by $\sigma_{i,j}$ the automorphism of the Fermat curve sending 
\[
 \sigma_{i,j}: (X:Y:Z) \mapsto (\zeta^i X :\zeta^j Y:Z).
\]
By computation we have 
\[
 \bar{\tau} \sigma_{i,j} \bar{\tau}^{-1}:(X:Y:Z) \mapsto (X:\zeta^jY: \zeta^i Z)=(\zeta^{-i}X: \zeta^{-i+j} Y: Z).
\]
Therefore, the conjugation action of $\bar{\tau}$ on 
$\mathrm{GL}(2,\Z/n\Z)=\A(\Z/n\Z \times \Z/n\Z)$ is given 
by the matrix
 $\begin{pmatrix}
   -1 & -1 \\
   0 & 1
  \end{pmatrix}$.

\subsection{The Heisenberg curve}
We will now describe the automorphisms of the Fermat curves $F_n$ that 
can be lifted to automorphisms of $C_n$. Every representative $\sigma \in \A(\mathbb{H})$  of an 
element   $\bar{\sigma} \in \A(F_n)=N_{\mathrm{SL}(2,\mathbb{R})}(\ker \psi)/\ker \psi$
should keep the group $\ker \psi$ invariant when acting by conjugation. 
The subgroup of the automorphism group of $F_n$ should  
also keep the group 
\[
 \mathrm{ker}(\phi)=\langle a^n,b^n,[a,b]^n \rangle 
\]
invariant, in order to extend to an automorphism of $C_n$.

The elements $a,b$ generating $\Z/n\Z \times \Z/n\Z$ modulo $\mathrm{ker}\psi$ keep both groups $\ker\phi,\ker\psi$ invariant, 
so $\Z/n\Z \times \Z/n\Z < \A(F_n)$ is lifted to a subgroup of automorphisms of $C_n$.  

Let $\sigma \in \Gamma$ be a representative of 
an element  $\bar{\sigma} \in S_3 \subset \A(F_n)$.
This element is lifted to an automorphism of the group $C_n$ 
if and only if the conjugation action of $\sigma$ keeps the defining group $\ker\phi$
invariant. 

Let $i_1,i_2$ be the involutions generating $S_3$ as defined in eq. (\ref{i-involution}).
Checking that the involution  $i_1$ keeps $\ker \phi$ invariant is trivial. 

We will now check the involution $i_2$:
It is clear that $(b^n)^{i_2}=b^{i_2}$ and
$[a,b]^{i_2}=[b^{-1},a^{-1}]$ so $\left( [a,b]^n \right)^{i_2}=[b^{-1},a^{-1}]^n$.

Let $a_H,b_H$ be the generators of the Heisenberg group, seen as elements in $F_2=\langle a_H,b_H \rangle / \ker \phi$ as given in eq. (\ref{a-b}). In order to check whether the generator $a^n$ of $\ker \phi$ are sent to $\ker \phi$ under the action of $i_2$ it is enough to prove that its image modulo $\ker \phi$ is the  zero element in the Heisenberg group. 
We compute 
\[
a_H^{i_2} =b_H^{-1} a_H^{-1} \mbox{ and } b_H^{i_2} =b_H.
\]
Therefore
\[
  a_H^{i_2}  =
 \begin{pmatrix}
  1 & 0 & 0 \\
  0 & 1 & -1 \\
  0 & 0 & 1 
 \end{pmatrix}
\begin{pmatrix}
  1 & -1 & 0 \\
  0 & 1 & 0 \\
  0 & 0 & 1 
 \end{pmatrix}=
 \begin{pmatrix}
  1 & -1 & 0 \\
  0 & 1 & -1 \\
  0 & 0 & 1 
 \end{pmatrix}
\]
so by eq. (\ref{power-comp}) we have 
\[
 \left( a_H^{i_2} \right)^n=
 \left\{
 \begin{array}{ll}
 \begin{pmatrix}
  1 & 0 & -\frac{n}{2} \\
  0 & 1 & 0 \\
  0 & 0 & 1 
 \end{pmatrix} & \mbox{ if } n\equiv 0 \mod 2 \\
 \begin{pmatrix}
  1 & 0 & 0 \\
  0 & 1 & 0 \\
  0 & 0 & 1 
 \end{pmatrix} & \mbox{ if } (n,2)=1
 \end{array}
 \right.
\]
and this gives the identity matrix if and only if $(n,2)=1$. Therefore in the case  $2 \mid n$ the element $(a^n)^{i_2}$ does not belong to the group $\ker \phi$.

 

\begin{lemma}
 The modular automorphism group for the Heisenberg curve is given by an extension 
 \[
  1 \rightarrow \Z/n\Z \rightarrow \A^{\mathrm{m}}(C_n) \rightarrow G_n \rightarrow 1
 \]
where $G_n$ is the group
 \[
 G_n=
 \left\{ 
 \begin{array}{ll}
  \A(F_n)& \mbox { if } (n,2)=1  \\
  (\Z/n\Z\times \Z/n\Z) \rtimes \Z/2\Z & \mbox{ otherwise}
  \end{array}
  \right.
 \]
\end{lemma}
\begin{proof}
The Heisenberg group is already in the automorphism group, fitting in the short exact sequence
\[
\xymatrix{
1 \ar[r] & \Z/n\Z \ar[r] \ar@{=} [d] 
& \mathrm{Aut}^{\mathrm{m}} (C_n) \ar[r] & G_n \ar[r] & 1 
\\
1    
\ar[r] & \Z/n\Z \ar[r]  & H_n \ar[r] \ar@{^{(}->}[u]
 & \Z/n\Z \times \Z/n\Z \ar[r] 
\ar@{^{(}->}[u]
 &
1
}
\]
We have computed the part of $\A(F_n)$ which lifts to automorphisms of $C_n$ according to 
 the value of $n$ modulo $2$. 

We will now prove that every modular 
 automorphism of $C_n$ restricts 
 to an automorphism of $F_n$. Indeed, using proposition \ref{rest-ext}  we have to show that every automorphism $\sigma$  that 
 fixes by conjugation elements of $\langle a^n,b^n, [a,b]^n\rangle$ fixes elements of $\langle a^n,b^n,[a,b]\rangle$ as well. 
 
 Assume that $\sigma$ fixes the group $\langle a^n,b^n, [a,b]^n\rangle$.
 If $(a^n)^\sigma$ is a word in $a^n,b^n,[a,b]^n$ then it is obviously a word 
 in $a^n,b^n,[a,b]$. For the commutator we will use the following result due to 
 Nielsen \cite[th. 3.9]{MagKarSol}
 \[
  [a,b]^\sigma= T [a,b]^{\pm 1} T^{-1},
 \]
where $T$ is a word in $a,b$. So the invariance of the commutator $[a,b]$ under 
the action of the outer automorphism $\sigma$ follows since $\langle a^n,b^n, [a,b]\rangle$
is a normal subgroup of $\langle a,b\rangle$.
\end{proof}
\begin{lemma}
For $2\mid n$ three points are ramified in $C_n\mapsto \mathrm{Aut}^{\mathrm{m}}(C_n)$, with ramification indices $4n,n,2$. 
\end{lemma}
\begin{proof}
In the cover $C_n \rightarrow F_n \rightarrow (\mathbb{Z}/n\Z \times \mathbb{Z}/n\Z) \backslash F_n \cong \mathbb{P}^1$
we have ramification on the points $P_{x=0},P_{x=1}$ and $P_\infty$, with ramification indices $n,n,2n$. 

Consider the group $S_3$ acting on $(\mathbb{Z}/n\Z \times \mathbb{Z}/n\Z) \backslash F_n$ and generated by the automorphisms $x\mapsto f(x)$, where 
\[
f(x) \in 
\left\{
x,1-x,\frac{1}{x}, \frac{1}{1-x}, \frac{x}{x-1}, 
\frac{x-1}{x}
\right\}.
\]
The involution $i_1:a \leftrightarrow b$ in terms of generators of the free group, corresponds to the  involution $x\mapsto 1-x$, which sends $1 \leftrightarrow 0$ and keeps $\infty, 1/2$ invariant. 
Therefore, the three ramified points have ramification indices $2,n,4n$. 
\end{proof}

\begin{theorem}
 Every automorphism of $C_n$ is modular, i.e., sends 
 the cusps to the cusps. In particular the automorphism group of the 
 curve $C_n$ in this case equals:
 \[
  \A(C_n)=N_{\mathrm{SL}(2,\mathbb{R})}(\ker \phi)/\ker \phi.
 \]
\end{theorem}
\begin{proof}
If $(n,2)=1$ then the group $D(n)$ corresponding to the Heisenberg curve in the triangle uniformization given in eq. (\ref{triangle}) is a normal subgroup of $\Delta(n,n,n)$. The normalizer $N_{\mathrm{PSL}(2,\mathbb{R})}(D(n))$ contains the trigonal curve $\Delta(n,n,n)$ and since every group containing a triangle group with finite index is also triangle \cite[th. 10.6.5 p.279]{beardon1995geometry} the normalizer is also a triangle group. By the computation of the modular automorphism group and the classification given in 
\cite[table 2]{Kallel2005-ru} we have that $N_{\mathrm{PSL}(2,\mathbb{R})}(D(n))=D(2,n,2n)$, which gives us
\[
\mathrm{Aut}(C_n) =\frac{N_{\mathrm{PSL}(2,\mathbb{R})}(D(n))}{D(n)}=\A^{\mathrm{m}} (C_n).
\]

For the $2\mid n$ case we will employ the Riemann-Hurwitz formula. Let $G$ be the automorphism group, $|G|=2n^3 m$, where $m=[G:\mathrm{Aut}^{\mathrm{m}}(C_n)]$. Let $Y=C_n^G$. Since $\mathbb{C}(Y) \subset C_n^{H_n}$ we have that $g_Y=0$. 
By eq. (\ref{g2Cn}) $2g_{C_n}-2 =n^2(n-3)+n^2/2$ and Riemann-Hurwitz theorem  \cite[ex. IV.2.5]{Hartshorne:77} gives us that 
\[
n^2(n-3)+n^2/2= n^32m(-2 +\sum_{i=1}^r  (1-1/e_i))
\]
where $e_i\geq 2$ are the ramification indices of the $r$-points of $\mathbb{P}^1$ ramified in extension $C_n \rightarrow G\backslash C_n=\mathbb{P}^1$. 
Set $\Omega_n=-2 +\sum_{i=1}^r  (1-1/e_i)$. Then the Riemann-Hurwitz formula can be written as 
\[
n-5/2= 2mn \Omega_n. 
\]
Observe that $\Omega_n\geq 0$ and $m\geq 1$. 
So for $n>3$ we must have $\Omega_n> 0$. If $\Omega_n\geq 1/4$,  then it is obvious that $m=1$.  
Indeed, we should have
\begin{equation}
\label{m-orig}
1\leq m =\frac{2n-5}{4n \Omega_n}\leq 2-5/n <2.
\end{equation}
In the above formula we have that for $n \leq 4$ the term $2-5/n<1$, which is not compatible with the $1\leq m$ inequality. This means that for $n=4$ the inequality $\Omega_n\geq 1/4$ is not possible so $\Omega_n<1/4$.
For $\Omega_4$ the ramification index index for one point is at least $16=4n$. 
\[
\Omega_4=-2+ 
\left(
1- \frac{1}{16e}
\right)
+
\sum_{i=2}^r
\left(
1-\frac{1}{e_i}
\right) \geq 
-\frac{17}{16} + \frac{r-1}{2}.
\]
The above value for $\Omega_4$ is negative for $r \leq 3$ and $1/4$ for $r=4$.  So we need less than $3$ ramification points, which we assume that have ramification indices $16e,\kappa,\lambda$. In this case we have
\[
\Omega_4=1-\frac{1}{16e} -\frac{1}{\ell} -\frac{1}{\kappa}.
\]
Since this value has to be smaller than $1/4$ the values for $e,\lambda,\kappa$ can not take very big values. If both $\kappa,\ell \geq 3$ then
\[
\Omega_4 \geq 1-\frac{1}{16}-\frac{2}{3} > \frac{1}{4}
\] 
which is impossible. So $\kappa=2$ and $\lambda=3$
is the only possible case. 
For this case $\Omega_4 \geq 5/48$ and eq. (\ref{m-orig}) implies that 
$
m \leq 9/5 <2.$

Now we consider the $n >4$ case and we will show that $\Omega_n \geq 1/4$.
If three or more points, other than the one with $e_1=4ne$, are ramified in the cover $C_n \rightarrow Y$, then 
\[
\Omega_n=-2 +1 -\frac{1}{4ne} + \sum_{i=2}^r 
\left(
1-\frac{1}{e_i}
\right)
\geq \frac{1}{2}-\frac{1}{4ne} \geq \frac{1}{4}.
\] 
Consider now that exactly three points are ramified in $C_n\rightarrow Y$. If two of them are ramified with ramification index $2$ then $\Omega_n=-1/4n$ and this is not allowed since $\Omega_n\geq 0$. 

Assume now that we have exactly $3$ ramification points with ramification $(4ne,\kappa,\ell)$. In this case
\[
\Omega_n\geq-2 + 
\left( 1 -\frac{1}{4n}
\right) + 
\left(
1-\frac{1}{\ell}
\right) +
\left(
1 -\frac{1}{\kappa}
\right)=1 -\frac{1}{4n} -\frac{1}{\ell} -\frac{1}{k}.
\]
If $n> 4$ then $n\geq 6$ ($n$ is even) so 
\[
\Omega_n \geq \frac{23}{24} -\frac{1}{\ell} -\frac{1}{k}.
\]
It is clear that the above quantity is bigger than $1/4$ if $\ell$ and $\kappa$ are big enough. For instance if $\ell,\kappa \geq 3$ then $23/24-2/3=7/24 $ and $\Omega_n\geq 1/4$. We have to check the case $\kappa=2$ and in this case
$\Omega_n=23/24-1/2-1/\ell=11/24-1/\ell$ so the inequality $\Omega_n \geq 1/4 $ holds provided $\ell \geq 5$.

The cases $\ell=3,4$ give the corresponding bounds
 $B_\ell=(2n-5) \Omega_\ell/(4n)$, 
\[
B_3=\frac{3 (2 n-5)}{n (3 n+2)}, \qquad 
B_4=\frac{2 n-5}{n^2+n}, \qquad
\]
All the above values are $<1$ and can not bound the quantity $m$, hence they can't occur.
\end{proof}

\section{The curve $C_3$}
The Fermat curve $F_3:x^3+y^3=1$ is elliptic and it has the projective cannonical Weierstrass form $zy^2=x^3-432z^3$, see \cite[p. 50-52]{MR1193029} and \cite[ex. 3 p.32]{MR2024529}. Its torsion $3$-points are the flexes which can be computed as the zeros of the Hessian determinant:
\[
\mathrm{Hess}(y^2z-x^3+432z^3)=
\det 
\begin{pmatrix}
-6x & 0 & 0 
\\
0 & 2z & 2y 
\\
0 & 2y & 2592 y
\end{pmatrix}
=
24(y^2-1296z^2)x.
\]
\begin{itemize}
	\item
If $x=0$, then $zy^2=-432z^3$ which gives the solutions
\[
(0:1:0), (0: 12 \sqrt{-3} :1), (0: -12\sqrt{-3} :1).
\] 
\item 
If $x\neq 0$, then $y^2=1296 z^2$, which gives the solutions
$(1:0:0)$, which does not satisfy the equation of the elliptic curve. We also have the solution  
$y=\pm \sqrt{1296} z$, which we plug into the equation of the elliptic curve to obtain:
\[
z^3 1296= x^3-432 z^3,
\] 
so for $z=1$ we obtain $x^3=1728$ so $x=12 \zeta_3^i$, and $\zeta_3=(-1+\sqrt{-3})/2$ is a primitive third root of unity.
We therefore have $6$-more $3$-torsion points namelly 
\[
(12 \zeta_3^i : \pm 36: 1 ).
\]
\end{itemize}
The curve $C_3$ is by Riemann-Hurwitz formula also an elliptic curve and the covering map 
\[
C_3 \rightarrow F_3
\]
is an isogeny. For each point of order $3$ computed above V\'elu method \cite{MR294345} can be applied and using sage \cite{sage8.9} we compute the following table:

\begin{center}
\newcolumntype{g}{>{\columncolor{Gray}}c}
\begin{tabular}{|g|c|c|}
\hline
\rowcolor{LightCyan}
 
point of order 3 & Equation of isogenus curve & $j$-invariant
 \\
\hline
$(0: \pm 12 \sqrt{-3} :1)$
& 
$y^2 = x^3 + 11664$
& 
$0$
\\
$(-6(1+\sqrt{-3}) : \pm 36 : 1)$
&
$y^2 = x^3 +  2160(1-\sqrt{-3})x -109296$
&  $-12288000$
\\
$( 6(1-\sqrt{-3}): \pm 36 : 1)$
&
$y^2 = x^3 + 2160(1+\sqrt{-3})x -109296$
&  $-12288000$
\\
$(12:\pm 36:1)$
& 
$y^2=x^3-4320x-109296$
& 
$-12288000$
\\
\hline
\end{tabular}
\end{center}
The three last curves have the same $j$-invariant and are isomorphic (they are quadratic twists of each other). The first one has $j$-invariant zero and therefore the automorphism group $\mathrm{Aut}^0(E)$ of $E$, consisted of automorphisms $E\rightarrow E$ which fix the identity, is a cyclic group of order $6$, see \cite[ch. III. par. 10.1]{MR2514094}.
This is compatible with the structure of the Heisenberg curve $C_3$, since the neutral element of $C_3$ is fixed by a group of order $6$. 

On the other hand the other 3 isomorphic curves of the above table have $j$-invariant $\neq 0,1728$ so $\mathrm{Aut}^0(E)$ is a cyclic group of order $2$. Therefore the equation of $C_3$ is given by 
\[
C_3: y^2 = x^3 + 2^4\cdot 3^6.
\]

 \def\cprime{$'$}
\providecommand{\bysame}{\leavevmode\hbox to3em{\hrulefill}\thinspace}
\providecommand{\MR}{\relax\ifhmode\unskip\space\fi MR }
\providecommand{\MRhref}[2]{%
  \href{http://www.ams.org/mathscinet-getitem?mr=#1}{#2}
}
\providecommand{\href}[2]{#2}

\end{document}